\documentclass[11pt]{article}
\usepackage{amsfonts}
\usepackage{amsmath}
\usepackage{amssymb}

\usepackage{todonotes}

\usepackage[utf8]{inputenc}

\title{On the uniform Rasmussen-Tamagawa \\ conjecture in the CM case}
\author{Davide Lombardo  \thanks{\texttt{davide.lombardo@math.u-psud.fr}} \\ \small{Département de Mathématiques d'Orsay}
}

\usepackage{newunicodechar}

\usepackage{hyperref}
\usepackage{url}
\usepackage[all,cmtip]{xy}

\usepackage{etoolbox}
\usepackage{amsthm}

\usepackage{a4wide}

\newtheorem{theorem}{Theorem}

\newtheorem{conjecture}[theorem]{Conjecture}
\newtheorem{corollary}[theorem]{Corollary}

\newtheorem{definition}[theorem]{Definition}

\newtheorem{lemma}[theorem]{Lemma}

\newtheorem{proposition}[theorem]{Proposition}

\theoremstyle{definition}
\newtheorem{remark}[theorem]{Remark}
\numberwithin{theorem}{section}

\newcommand{\abGal}[1] {\operatorname{Gal}\big(\overline{#1}/#1\big)}
\newcommand{\RTEndo}{\operatorname{RT}^{\operatorname{CM, 1}}}
\newcommand{\RTAll}{\operatorname{RT}^{\operatorname{CM, 2}}}
\newcommand{\CAll}{C^{(2)}}
\newcommand{\CEndo}{C^{(1)}}

\date{}
\begin{document}

\maketitle
\begin{abstract}
We prove a uniform version of a finiteness conjecture due to Rasmussen and Tamagawa in the case of CM abelian varieties. This extends the result of \cite{zbMATH06443603} from elliptic curves to abelian varieties of arbitrary dimension.
\end{abstract}

\section{Introduction}
Motivated by previous work of Anderson and Ihara \cite{MR960948}, in \cite{MR2470396} and \cite{2013arXiv1302.1477R} Rasmussen and Tamagawa have formulated (and partially proven) a series of finiteness conjectures for abelian varieties $A$ over number fields $K$ such that the extension $K(A[\ell^\infty])/K(\mu_{\ell^\infty})$ is both pro-$\ell$ and unramified away from $\ell$. The strongest form of their conjecture, as stated in \cite[Conj.~2]{2013arXiv1302.1477R}, is the following uniform finiteness statement:
\begin{conjecture}\label{conj_UniformRT}
Let 
\[
\operatorname{RT}\left(K,g,\ell \right) = \left\{A\text{ abelian variety over }K \bigm\vert \begin{matrix} \dim A= g \\ K(A[\ell^\infty])/K(\mu_{\ell^\infty}) \text{ is pro-}\ell\text{ and} \\ \text{unramified outside }\ell \end{matrix} \right\}.
\]
There is a function $B(n,g)$ such that, for every number field $K$ of degree $n$ and every prime $\ell>B(n,g)$, the set $\operatorname{RT}\left(K,g,\ell \right)$ is empty.
\end{conjecture}
Much progress has been made on this conjecture -- in particular, Rasmussen and Tamagawa themselves have proven \cite{2013arXiv1302.1477R} that the Generalized Riemann Hypothesis implies conjecture \ref{conj_UniformRT} for $n$ odd -- but an unconditional proof is only known for $g=1$ and $[K:\mathbb{Q}]$ equal to either 1 or 3. More recently, Bourdon \cite{zbMATH06443603} has given an unconditional proof of a similar finiteness result for CM elliptic curves over arbitrary number fields:
\begin{theorem}{(Bourdon \cite{zbMATH06443603})}\label{thm_Abbey}
Let $K$ be a number field with $[K:\mathbb{Q}]=n$. There is a constant $C=C(n)$ depending only on $n$ with the following property: if
there exists a CM elliptic curve $E/K$ with $K(E[\ell^\infty])$ a pro-$\ell$ extension of
$K(\mu_\ell)$ for some rational prime $\ell$, then $\ell \leq C$.
\end{theorem}

The purpose of this note is to extend theorem \ref{thm_Abbey} to CM abelian varieties of arbitrary dimension. To be more precise, by an \textbf{abelian variety of CM type} over $K$ we mean an abelian variety $A/K$ such that $\operatorname{End}_{\overline{K}}(A) \otimes \mathbb{Q}$ contains an étale $\mathbb{Q}$-algebra of dimension equal to $2\dim A$. 
We shall show the following higher-dimensional analogue of theorem \ref{thm_Abbey}:
\begin{theorem}\label{thm_RT}
Let
\[
\operatorname{RT}^{\operatorname{CM}}\left(K,g,\ell \right) = \left\{A\text{ abelian variety over }K \bigm\vert \begin{matrix} \dim A= g \\ A\text{ of CM type}\\ K(A[\ell^\infty])/K(\mu_\ell) \text{ is pro-}\ell \end{matrix} \right\}.
\]
There exists a function $C(n,g)$ with the following property: for every number field $K$ of degree at most $n$ the set $\operatorname{RT}^{\operatorname{CM}}(K,g,\ell)$ is empty for all $\ell > C(n,g)$. 
\end{theorem}

As it is clear, theorem \ref{thm_RT} yields a proof of conjecture \ref{conj_UniformRT} in the special case of CM abelian varieties. 
Notice that since CM abelian varieties acquire good reduction everywhere over a finite extension of their field of definition, and this extension can be taken of degree bounded by a constant depending only on the dimension, the condition that $K(A[\ell^\infty])/K$ be unramified outside $\ell$ is inessential in the CM case. In general, however, we do not expect finiteness if we both drop the ramification requirement and leave the realm of CM abelian varieties.

We conclude this brief introduction with a quick overview of the material to be covered in this article. In section \ref{sect_Preliminaries} we show that in order to prove theorem \ref{thm_RT} one only needs to deal with geometrically simple abelian varieties with multiplication by the full ring of integers of the corresponding CM field. In §\ref{sect_LowerBound} we recall a lower bound on the degree of the division fields of CM abelian varieties (taken from \cite{2015arXiv150604734L}), while in §\ref{sect_Tsimerman} we show how a recent theorem of Tsimerman \cite{2015arXiv150601466T} gives a finiteness result for the set of CM fields that can act on $g$-dimensional CM abelian varieties defined over fields of degree at most $n$. In §\ref{sect_Conclusion} we finish the proof of theorem \ref{thm_RT}, while §\ref{sect_6} contains a few remarks on the problem of effectivity, together with a more detailed study of the case $n=1$, $g=2$.

\section{Preliminary reductions}\label{sect_Preliminaries}
The situation is simpler if we assume that our abelian varieties have all their endomorphisms defined over $K$. It is thus natural to consider the following subset of $\operatorname{RT}^{\operatorname{CM}}(K,g,\ell)$:
\[
\RTEndo\left(K,g,\ell \right) = \left\{A \in \operatorname{RT}^{\operatorname{CM}}(K,g,\ell) \bigm\vert \operatorname{End}_{\overline{K}}(A)=\operatorname{End}_K(A)\right\}.
\]
Fortunately, as the following lemma shows, not much is lost in considering $\RTEndo\left(K,g,\ell \right)$ instead of the full set $\operatorname{RT}^{\operatorname{CM}}\left(K,g,\ell \right)$: 

\begin{lemma}\label{lemma_reduction1}
Suppose there exists a function $\CEndo(n,g)$ with the following property: for every number field $K$ of degree at most $n$, the set $\RTEndo\left(K,g,\ell \right)$ is empty for all $\ell > \CEndo(n,g)$. Then theorem \ref{thm_RT} holds.
\end{lemma}
\begin{proof}
Recall that, for fixed $g$, there is a constant $D(g)$ with the following property: for every abelian variety $A$ of dimension $g$ over a number field $K$ there exists an extension $F$ of $K$, of degree at most $D(g)$, such that $\operatorname{End}_{\overline{K}}(A)=\operatorname{End}_F(A)$ (sharp bounds for $D(g)$ can be found in \cite{MR1154704}). Set $C(n,g)= \CEndo(D(g) \cdot n,g)$.

Let now $K$ be a number field of degree at most $n$. If $A/K$ is an element of $\operatorname{RT}^{\operatorname{CM}}\left(K,g,\ell \right)$, then we can find a number field $F$ such that $[F:\mathbb{Q}] =[F:K][K:\mathbb{Q}] \leq D(g) n$ and $\operatorname{End}_F(A)=\operatorname{End}_{\overline{F}}(A)$. The abelian variety $A/F$ is then an element of $\RTEndo\left(F,g,\ell \right)$, which by assumption is empty for $\ell > \CEndo(D(g)n,g)$. This clearly implies that $\operatorname{RT}^{\operatorname{CM}}\left(K,g,\ell \right)$ is empty as long as $\ell > \CEndo(D(g)n,g)=C(n,g)$.
\end{proof}

We can also restrict ourselves to geometrically simple varieties:

\begin{lemma}\label{lemma_reduction2}
Let 
\[
\RTAll(K,g,\ell)=\left\{ A \in \RTEndo(K,g,\ell) \bigm\vert A \text{ is absolutely simple} \right\}
\]
and suppose there is a function $\CAll(n,g)$ with the following property: for every number field $K$ of degree at most $n$, the set $\RTAll\left(K,g,\ell \right)$ is empty for all $\ell > \CAll(n,g)$. Then theorem \ref{thm_RT} holds.
\end{lemma}
\begin{proof}
It suffices to show that there exists a function $\CEndo(n,g)$ as in lemma \ref{lemma_reduction1}. We claim that we can take $\CEndo(n,g)=\max_{g' \leq g} \CAll(n,g')$. To see this, suppose by contradiction that there exists a number field $K$ of degree at most $n$ and a prime $\ell>\max_{g' \leq g} \CAll(n,g')$ such that $\RTEndo(K,g,\ell)$ is nonempty. Let $A/K$ be an element of this set. By definition we have $\operatorname{End}_{\overline{K}}(A)=\operatorname{End}_K(A)$, so all the abelian subvarieties of $A$ are defined over $K$. Let $A'/K$ be an absolutely simple subvariety of $A/K$, and let $g'$ be its dimension. It is clear that $A'$ has complex multiplication, and that the extension $K(A'[\ell^\infty])/K(\mu_\ell)$ is pro-$\ell$ since it is a sub-extension of the (pro-$\ell$) extension $K(A[\ell^\infty])/K(\mu_\ell)$. It follows that $A'$ is an element of $\RTAll(K,g',\ell)$, but this is a contradiction, because by assumption $\RTAll(K,g',\ell)$ is empty for $\ell>\CAll(K,g',\ell)$. 
\end{proof}

\section{The case $\operatorname{End}_{\overline{K}}(A)=\operatorname{End}_K(A)$}\label{sect_LowerBound}
It remains to establish the existence of a function $\CAll(n,g)$ as in lemma \ref{lemma_reduction2}. A key step in doing so is the following proposition:
\begin{proposition}\label{prop_Key}
Let $A/K$ be an element of $\RTAll(K,g,\ell)$ and let $R=\operatorname{End}_{\overline{K}}(A)$.
Either $\ell$ is at most $[K:\mathbb{Q}] (g+2)^{3(g+1)}$ or it divides the discriminant of $E:=R \otimes \mathbb{Q}$.
\end{proposition}
\begin{proof}
We shall suppose from the start that $\ell$ does not divide the discriminant of $E$, that is, that $\ell$ is unramified in $E$, and prove the claimed bound.
Consider the tower of extensions $K(A[\ell^\infty])/K(A[\ell])/K(\mu_\ell)$. Since by assumption $K(A[\ell^\infty])/K(\mu_\ell)$ is pro-$\ell$, this holds a fortiori for the extension $K(A[\ell])/K(\mu_\ell)$. 

On the other hand, the hypothesis $\operatorname{End}_{\overline{K}}(A)=\operatorname{End}_K(A)$ entails that the action of $\abGal{K}$ on $A[\ell]$ factors through $\abGal{K}^{ab} \to (R \otimes \mathbb{F}_\ell)^\times$ (\cite[Corollary 2 to Theorem 5]{MR0236190}). Since $\ell$ is unramified in $E$ we see that $(R \otimes \mathbb{F}_\ell)^\times \subseteq \left(\mathcal{O}_E \otimes \mathbb{F}_\ell \right)^\times$ has order prime to $\ell$, hence the same is true for $G_\ell:=\operatorname{Gal}\left(K(A[\ell])/K \right)$. Since on the other hand $K(\mu_\ell)/K$ is a sub-extension of $K(A[\ell])/K$, and by hypothesis $\operatorname{Gal}\left( K(A[\ell]) / K(\mu_\ell) \right)$ is an $\ell$-group, this implies $K(A[\ell])=K(\mu_\ell)$. 

Also notice that the Mumford-Tate group of $A$ is a subtorus of $\operatorname{Res}_{E/\mathbb{Q}}(\mathbb{G}_{m,E})$, which has good reduction at $\ell$ by the Galois criterion: in particular, $\operatorname{MT}(A)$ defines a torus over $\mathbb{F}_\ell$, and the Galois group $G_\ell$ is a subgroup of $\operatorname{MT}(A)(\mathbb{F}_\ell)$. Notice furthermore that the degree $\left[K(\mu_\ell):K\right]$ is at most $\varphi(\ell)=\ell-1$.

We now give a lower bound for the degree $K(A[\ell])/K$. We take the notation of \cite{2015arXiv150604734L}: we denote by $r$ the rank of $\operatorname{MT}(A)$, by $\mu$ the number of roots of unity in $E$, by $E^*$ the reflex field of $E$, and by $T_E$ (resp.~$T_{E^*}$) the algebraic group $\operatorname{Res}_{E/\mathbb{Q}}(\mathbb{G}_{m,E})$ (resp.~$\operatorname{Res}_{E^*/\mathbb{Q}}(\mathbb{G}_{m,E^*})$). Finally, we denote by $F$ the group of connected components of $\ker\left( T_{E^*} \xrightarrow{N} T_{E} \right)$, where $N$ is the reflex norm. Since $G_\ell \subseteq \operatorname{MT}(A)(\mathbb{F}_\ell)$ and $\ell$ is unramified in $E$, we see by \cite[Theorems 1.2 and 1.3]{2015arXiv150604734L} that the degree of $K(A[\ell])/K$ is at least
\[
\begin{aligned}
\frac{1}{[\operatorname{MT}(A)(\mathbb{F}_\ell):G_\ell]} \left| \operatorname{MT}(A)(\mathbb{F}_\ell) \right| & \geq \frac{(1-1/\ell)^r \ell^{r} }{\mu \cdot [K:E^*] \cdot |F|^{2r}}.
\end{aligned}
\]
We now give (rough) estimates for the various terms appearing in this expression:
\begin{itemize}
\item the degree $\displaystyle [K:E^*]=\frac{[K:\mathbb{Q}]}{[E^*:\mathbb{Q}]}$ does not exceed $\frac{1}{2}[K:\mathbb{Q}]$;
\item the number $\mu$ of roots of unity in $E$ satisfies $\varphi(\mu)\leq [E:\mathbb{Q}]=2g$; since $\varphi(x)\geq \frac{\sqrt{x}}{2}$ for all positive integers $x$, we have $\mu \leq (4g)^2$;
\item again by \cite[Theorem 1.3]{2015arXiv150604734L} we have $|F| \leq 2\displaystyle \left(\frac{r+1}{4}\right)^{(r+1)/2}$.
\end{itemize}
Putting everything together we find
$
\displaystyle 
\left[K(A[\ell]):K\right] 
 \geq \frac{2^{2r^2+1}}{16g^2} \cdot \frac{(\ell-1)^r}{[K:\mathbb{Q}]}  (r+1)^{-r(r+1)}.
$

\smallskip

A theorem of Ribet \cite[Formula (3.5)]{MR608640} yields the inequality $r \geq 2+\log_2(g)$, so that we have 
$
2^{2r^2+1} \geq 2^9g^2;
$
we thus obtain the inequality
$
\displaystyle \left[K(A[\ell]):K\right] \geq 2^5 \frac{ (\ell-1)^r}{[K:\mathbb{Q}]} (r+1)^{-(r+1)} ,
$
which, combined with $\left[K(A[\ell]):K\right]=[K(\mu_\ell):K] \leq \ell-1$, leads to
$
\displaystyle  \ell-1 \geq 2^5 \cdot \frac{(\ell-1)^r}{[K:\mathbb{Q}]} (r+1)^{-r(r+1)},
$
and finally to
\[
\ell-1 \leq \left(\frac{[K:\mathbb{Q}]}{32}\right)^{1/(r-1)} \cdot (r+1)^{r(r+1)/(r-1)} < [K:\mathbb{Q}] (r+1)^{3r} \leq [K:\mathbb{Q}] (g+2)^{3(g+1)}
\]
as claimed (notice that $\ell-1 < [K:\mathbb{Q}] (g+2)^{3(g+1)}$ implies $\ell \leq [K:\mathbb{Q}] (g+2)^{3(g+1)}$).
\end{proof}
\begin{remark}
As it is clear from the proof, one can obtain much sharper inequalities for large $g$: for example, as long as $g \geq 2$, we have $r \geq 3$ by Ribet's inequality, and in the very last step of the previous proof we obtain 
$
\ell-1 \leq [K:\mathbb{Q}]^{1/2} (r+1)^{2r}.
$
\end{remark}

\section{A theorem of Tsimerman}\label{sect_Tsimerman}
To finish the proof of theorem \ref{thm_RT} we shall need a way to control the possible endomorphism algebras of CM abelian varieties of a given dimension. This is made possible by corollary \ref{cor_Greenberg} below, which is in turn a consequence of a recent result of Tsimerman (theorem \ref{thm_Tsimerman}).
\begin{definition}
Let $A/\overline{\mathbb{Q}}$ be an abelian variety of CM type. The field of moduli of $A$ is the intersection of all the number fields $F$ such that there exists an abelian variety $A_F$ over $F$ that satisfes $(A_F)_{\overline{\mathbb{Q}}}=A$.
\end{definition}
\begin{theorem}{(\cite[Theorem 1.1]{2015arXiv150601466T})}\label{thm_Tsimerman}
For every positive $g$ there exist constants $k_g, \delta_g > 0$ such that if $E$ is a CM field of degree $2g$ and if $A$ is any abelian variety $\overline{\mathbb{Q}}$ of dimension $g$ with endomorphism
ring equal to the full ring of integers $\mathcal{O}_E$ of $E$, then the field of moduli $F$ of $A$ satisfies
\[
[F : \mathbb{Q}] \geq k_g |\operatorname{disc}(E)|^{\delta_g}.
\]
\end{theorem}

\begin{corollary}\label{cor_Greenberg}
Let $n, g$ be fixed positive integers. Consider the set $\mathcal{A}(n,g)$ all $g$-dimensional, geometrically simple abelian varieties $A/K$ of CM type, where $K$ is a number field of degree at most $n$. The set
\[
\mathcal{R}(n,g) = \left\{ \operatorname{End}_{K}(A) \otimes \mathbb{Q} \bigm\vert A \in \mathcal{A}(n,g) \right\}
\]
is finite.
\end{corollary}
\begin{proof}
Consider an abelian variety $A \in \mathcal{A}(n,g)$ with field of definition $K$, and let $E$ denote $\operatorname{End}_{\overline{K}}(A) \otimes_{\mathbb{Z}} \mathbb{Q}$; as it is well-known, $A$ is $K$-isogenous to an abelian variety $A'/K$ with multiplication by the full ring of integers of $E$. Let $F$ be the field of moduli of $A'$. Clearly we have $[K:\mathbb{Q}] \geq [F:\mathbb{Q}]$, and applying the previous theorem we find
\[
n \geq [K:\mathbb{Q}] \geq [F:\mathbb{Q}] \geq k_g |\operatorname{disc}(E)|^{\delta_g};
\]
in particular, $\operatorname{disc}(E)$ is bounded (in absolute value), hence there are only finitely many possibilities for $\operatorname{End}_{\overline{K}}(A') \otimes \mathbb{Q}=\operatorname{End}_{\overline{K}}(A) \otimes \mathbb{Q}$. As $\operatorname{End}_K(A) \otimes \mathbb{Q}$ is a subfield of $\operatorname{End}_{\overline{K}}(A) \otimes \mathbb{Q}$, this finishes the proof.
\end{proof}
\begin{remark}\label{rmk_EC}
The case $g=1$ (that is, the case of elliptic curves) of this corollary is well known, and is also a key ingredient for the arguments of \cite{zbMATH06443603}. To see why the case $g=1$ follows from the classical theory of elliptic curves, consider all number fields $K$ of degree at most $n$, and all elliptic curves $E_1/K$ with (potential) complex multiplication. If $E_1/K$ is such an elliptic curve, with complex multiplication by an order $R$ in the quadratic imaginary field $F$, then the action of $F$ on $E_1$ is defined over the compositum $FK$, and we can find an elliptic curve $E_2/FK$, isogenous to $E_1$ over $FK$, that has full complex multiplication by the ring of integers of $F$. Now it is well-known that the $j$-invariant of $E_2$ generates the Hilbert class field $H$ of $F$, and on the other hand $j(E_2)$ is in $FK$ by assumption, so it follows that
\[
h(F)=[H:F] \leq [FK:\mathbb{Q}] \leq 2[K:\mathbb{Q}] \leq 2n
\]
is bounded by $n$ alone. The Brauer-Siegel theorem implies that there are only finitely many imaginary quadratic fields $F$ with $h(F) \leq 2n$, and the finiteness of $\mathcal{R}(n,1)$ follows.
\end{remark}

\section{Conclusion}\label{sect_Conclusion}
We are now ready to prove theorem \ref{thm_RT}:
\begin{theorem}
There exists a function $C(n,g)$ such that $\operatorname{RT}^{\operatorname{CM}}(K,g,\ell)$ is empty for all number fields $K$ of degree at most $n$ and all primes $\ell > C(n,g)$.
\end{theorem}
\begin{proof}
According to lemma \ref{lemma_reduction2}, it suffices to show the existence of a function $\CAll(n,g)$ such that $\RTAll(K,g,\ell)$ is empty for all number fields $K$ of degree at most $n$ and for all $\ell>\CAll(n,g)$. Consider the set $\mathcal{R}(n,g)$ of corollary \ref{cor_Greenberg}, and let $\Delta$ be the maximum of the discriminants $|\operatorname{disc}(E)|$ for $E$ varying in $\mathcal{R}(n,g)$.

We claim that we can take 
$
\CAll(n,g)=\max\left\{\Delta, n(g+2)^{3(g+1)}\right\}.
$ 
To see this, consider a number field $K$ of degree at most $n$ and an element $A/K$ of $\RTAll\left(K,g,\ell \right)$, and set $E=\operatorname{End}_K(A) \otimes \mathbb{Q}$. By proposition \ref{prop_Key}, we have either $\ell \leq n(g+2)^{3(g+1)} \leq \CAll(n,g)$ or $\ell \leq |\operatorname{disc}(E)| \leq \Delta \leq \CAll(n,g)$; in particular, $\RTAll(K,g,\ell)$ is empty for $\ell > \CAll(n,g)$ as claimed.
\end{proof}

\section{Some remarks on effectivity}\label{sect_6}
Unlike theorem \ref{thm_Abbey}, our theorem \ref{thm_RT} is unfortunately non-effective: the source of this can be traced back to the proof of theorem \ref{thm_Tsimerman}, and more specifically to Corollary 3.2 of \cite{2015arXiv150601466T}, whose proof depends on the full strength of the Brauer-Siegel theorem, which is not known to be effective at present. Notice that other parts of Tsimerman's argument also require the Brauer-Siegel theorem, but they can be made effective by using the results of \cite{MR0342472}, so \cite[Corollary 3.2]{2015arXiv150601466T} is really the crux of the matter. By contrast, notice that the proof of the case $g=1$ of corollary \ref{cor_Greenberg} sketched in remark \ref{rmk_EC} \textit{can} be made effective: as it is well known, the problem of determining all imaginary quadratic fields of a given class number can be solved effectively. This fact is exploited in \cite{zbMATH06443603} to produce explicit bounds for the function $C(n,1)$ for various values of $n$.

On the other hand, even if one is willing to assume the truth of the Generalized Riemann Hypothesis (which -- as it is well known -- implies effective versions of the Brauer-Siegel theorem), the argument of \cite{2015arXiv150601466T} gives for the constant $\delta_g$ of theorem \ref{thm_Tsimerman} a very small value, intimately tied to a certain exponent appearing in the so-called Isogeny Theorem of Masser and W\"ustholz \cite{MR1207211} \cite{MR1217345}; the Brauer-Siegel theorem is only used to determine the value of $k_g$, and has no influence on $\delta_g$. 
Using the (currently) best available isogeny bound, due to Gaudron and Rémond \cite{PolarisationsEtIsogenies}, we see for example that theorem \ref{thm_Tsimerman} holds for all values of $\delta_2$ strictly smaller than $2^{-16}$: clearly this number is so small that it makes it impossible in practice to use theorem \ref{thm_Tsimerman} to determine the set $\mathcal{R}(n,g)$. Conditionally on GRH, sharper results are known, but none of them seems to be completely explicit at present: in the context of giving lower bounds on Galois orbits of special points on Shimura varieties, Tsimerman, Ullmo and Yafaev have proven various lower bounds on the degree of the field of moduli of a CM abelian variety (cf.~for example \cite{MR3323347} and \cite{MR2947946}), but their results contain some non-explicit constants that seem hard to compute in practice.

Slightly different techniques -- mainly coming from classical analytic number theory -- can however yield results on the sets $\mathcal{R}(n,g)$ for certain small values of $g$ and $n$, which in turn allows us to determine an admissible value for $C(n,g)$ -- and sometimes even the optimal value -- via the argument described in the previous sections. For example, we can prove

\begin{proposition}
We can take $C(1,2)=163$, and this value is optimal.
\end{proposition}
\begin{proof}
It is clear by definition that we must have $C(1,2) \geq C(1,1) \geq 163$, where the optimal value of $C(1,1)=163$ is taken from \cite{MR2470396} (see also \cite{zbMATH06443603}). 
Consider now an abelian surface $A/\mathbb{Q}$ admitting potential complex multiplication, and suppose first that $A_{\overline{\mathbb{Q}}}$ is isogenous to the product of two elliptic curves. Let $\ell$ be a prime larger than 163: we claim that $\mathbb{Q}(A[\ell])/\mathbb{Q}(\mu_\ell)$ cannot be pro-$\ell$. Suppose the contrary: we shall obtain a contradiction. We shall need to rely on the results of \cite{MR2982436}, so we take the notation of that paper for the ``Galois type'' of our abelian variety $A$.
Let $K$ be a minimal field of definition for the endomorphisms of $A$; by \cite{MR2982436}, we have $[K:\mathbb{Q}] \bigm\vert 48$, and $K$ is contained in $\mathbb{Q}(A[\ell])$ by \cite[Propositions 2.2 and 2.3]{MR1154704}. In fact we know even more, namely that $K/\mathbb{Q}$ is normal, with Galois group isomorphic to a subgroup of either $S_4 \times \mathbb{Z}/2\mathbb{Z}$ or $D_6 \times \mathbb{Z}/2\mathbb{Z}$ (\cite[Table 8]{MR2982436}).
Consider now the following diagram of field extensions:

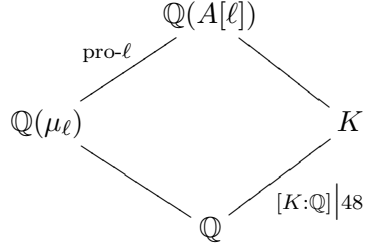
\[
\xymatrix{
& \mathbb{Q}(A[\ell]) \ar@{-}[ld]_{\text{pro-}\ell} \ar@{-}[rd] \\
\mathbb{Q}(\mu_\ell) \ar@{-}[rd]
&  
&K \ar@{-}[ld]^{[K:\mathbb{Q}] \bigm\vert 48}\\
&\mathbb{Q}}
\]

Let $G_K$ (resp.~$G_{\mathbb{Q}(\mu_\ell)}$, $G_\mathbb{Q}$) be the Galois group of $\mathbb{Q}(A[\ell])$ over $K$ (resp.~$\mathbb{Q}(\mu_\ell)$, $\mathbb{Q}$). Then $[K:\mathbb{Q}]=[G_\mathbb{Q}:G_K]$ is prime to $\ell$, hence $G_K$ contains a maximal $\ell$-Sylow subgroup of $G_\mathbb{Q}$. On the other hand, $G_{\mathbb{Q}(\mu_\ell)}$ is a maximal Sylow subgroup of $G_\mathbb{Q}$ (notice that $\ell \nmid [\mathbb{Q}(\mu_\ell):\mathbb{Q}])$, and it is normal in $G_\mathbb{Q}$ because $\mathbb{Q}(\mu_\ell)$ is Galois over $\mathbb{Q}$. Now since all the maximal $\ell$-Sylow subgroups of a group are conjugate to each other, this proves that $G_\mathbb{Q}$ has a unique maximal $\ell$-Sylow, namely $G_{\mathbb{Q}(\mu_\ell)}$. It follows that $G_K$ contains $G_{\mathbb{Q}(\mu_\ell)}$, hence that $K$ is contained in $\mathbb{Q}(\mu_\ell)$. In particular, $K/\mathbb{Q}$ is a cyclic extension, and since its Galois group is a subgroup of either $S_4 \times \mathbb{Z}/2\mathbb{Z}$ or $D_6 \times \mathbb{Z}/2\mathbb{Z}$ the group $\operatorname{Gal}(K/\mathbb{Q})$ must be cyclic of order $1$, $2$, $3$, $4$ or $6$. 
Depending on whether the simple factors of $A_{\overline{\mathbb{Q}}}$ are isogenous or not, the following are then the only possibilities for the Galois type of $A$:
\begin{enumerate}
\item $A_{\overline{\mathbb{Q}}}$ is isogenous to the square of an elliptic curve: then by what we have just proved, combined with \cite[Table 8]{MR2982436}, the Galois type of $A$ is one of $\mathbf{F}[C_n]$ ($n \in \{1,2,3,4,6\}$), $\mathbf{F}[C_2,C_1,\mathbb{H}]$, $\mathbf{F}[C_2,C_1,\operatorname{M}_2(\mathbb{R})]$, $\mathbf{F}[C_4,C_2]$, $\mathbf{F}[C_6,C_3,\operatorname{M}_2(\mathbb{R})]$, $\mathbf{F}[C_6,C_3,\mathbb{H}]$;
\item the two elliptic curves appearing as simple factors of $A_{\overline{\mathbb{Q}}}$ are non-isogenous: then the Galois type of $A$ is one of $\mathbf{D}[C_1]$, $\mathbf{D}[C_2,\mathbb{R}\times\mathbb{C}]$, $\mathbf{D}[C_2,\mathbb{R}\times\mathbb{R}]$, $\mathbf{D}[C_4]$.
\end{enumerate}

We claim that there exists a quadratic extension $M$ of $\mathbb{Q}$ such that $A_M$ admits a 1-dimensional abelian subvariety defined over $M$ (equivalently, $A_M$ is $M$-isogenous to the product of two elliptic curves defined over $M$).

Case (2) is easy to deal with: according to \cite[Theorem 4.3]{MR2982436}, only type $\mathbf{D}[C_4]$ can arise for an abelian surface $A$ defined over $\mathbb{Q}$, and in this case $A_{\overline{\mathbb{Q}}}$ is simple (\cite[§4.3 and 4.4]{MR2982436}), contrary to our assumptions. We can therefore focus on case (1). 

Let us first notice that, among the various subcases of (1), only cases $\mathbf{F}[C_2,C_1,\operatorname{M}_2(\mathbb{R})]$ and $\mathbf{F}[C_6,C_3,\operatorname{M}_2(\mathbb{R})]$ can arise for $A$ defined over $\mathbb{Q}$ (\cite[Theorem 4.3]{MR2982436}). As for these two Galois types, the claim about the existence of $M$ is obvious for $\mathbf{F}[C_2,C_1,\operatorname{M}_2(\mathbb{R})]$, because in this case $K$ is itself a quadratic extension of $\mathbb{Q}$, and since all the endomorphisms of $A$ are defined over $K$, so are its abelian subvarieties. For case $\mathbf{F}[C_6,C_3,\operatorname{M}_2(\mathbb{R})]$, we know by \cite[§4.5.2]{MR2982436} that $\operatorname{End}_{K}(A) \otimes \mathbb{R} \cong \operatorname{M}_2(\mathbb{C})$, and the action of $\operatorname{Gal}\left(K/\mathbb{Q}\right) \cong \mathbb{Z}/6\mathbb{Z}$ on it is determined by the fact that there is a generator $g$ of $\mathbb{Z}/6\mathbb{Z}$ that acts on $2 \times 2$ complex matrices by the formula
$
\displaystyle
\left(\begin{matrix} a & b \\ c & d \end{matrix} \right) \mapsto \left(\begin{matrix} \bar{d} & \zeta_3 \bar{c} \\ \overline{\zeta_3} \bar{b} & \bar{a} \end{matrix} \right).
$
It follows that $g^2$ acts as 
$
\left(\begin{matrix} a & b \\ c & d \end{matrix} \right) \mapsto \left(\begin{matrix} a & \zeta_3^2 b \\ \overline{\zeta_3}^2 c & d \end{matrix} \right),
$
so the fixed ring of $g^2$ is isomorphic to $\mathbb{C}\times \mathbb{C}$ (matrices with $b=c=0$). If we denote by $M$ the fixed field of $g^2$, then $M/\mathbb{Q}$ is a quadratic extension, and $\operatorname{End}(A_M) \otimes \mathbb{R} \cong \operatorname{M}_2(\mathbb{C})^{\langle g^2 \rangle}=\mathbb{C}\times \mathbb{C}$. Since by assumption $\operatorname{End}_M(A)$ cannot be a number field of degree 4, it follows that $\operatorname{End}_M(A)$ is not an integral domain, hence that $A_M$ is nonsimple as claimed.

Let now $A_1/M$ be an elliptic curve contained in $A_M$: the extension $M(A_1[\ell])/M(\mu_\ell)$, being contained in $M(A[\ell])/M(\mu_\ell)$, is pro-$\ell$, but by definition of $C(2,1)$ this is impossible for $\ell>C(2,1)=163$, which finishes the proof in this case.

Consider then the case of $A$ being geometrically simple. According to \cite[Theorems 2.1 and 2.2]{MR2430994} (see also \cite{MR1370204}, \cite{MR1901990}, and \cite{MR1807666}), if $p$ is a prime ramified in $E:=\operatorname{End}_{\overline{\mathbb{Q}}}(A) \otimes \mathbb{Q}$, then $p \leq 61$. We claim that $\mathbb{Q}(A[\ell])/\mathbb{Q}(\mu_\ell)$ cannot be pro-$\ell$ for any prime $\ell > 61$. Indeed, let $\ell>61$ be a prime, and let $(E^*,\Phi^*)$ be the reflex type of $(E,\Phi)$, where $(E,\Phi)$ is the CM type attached to $A$. It is well-known that all endomorphisms of $A$ are defined over $E^*$ (cf.~\cite[Chapter 3, Theorem 1.1]{MR713612}), and clearly if the extension $\mathbb{Q}(A[\ell])/\mathbb{Q}(\mu_\ell)$ is pro-$\ell$ then the same is true for $E^*(A[\ell])/E^*(\mu_\ell)$. As in the proof of proposition \ref{prop_Key}, since all the endomorphisms of $A$ are defined over $E^*$ we know that the representation $\abGal{E^*} \to \operatorname{Aut} A[\ell]$ factors through $(\mathcal{O}_E \otimes \mathbb{F}_\ell)^\times$, which is a group of order prime to $\ell$ since $\ell$ is unramified in $E$. It follows that $E^*(A[\ell])=E^*(\mu_\ell)$, hence $[E^*(A[\ell]):E^*] \leq \ell-1$. 

Observe now that (in the notation of the proof of proposition \ref{prop_Key}) we have $|F|=1$ and $r=3$, because this is true for all absolutely simple CM abelian surfaces; we then obtain from \cite[Theorems 1.2 and 1.3]{2015arXiv150604734L} the inequality $[E^*(A[\ell]):E^*] \geq \frac{1}{|\mu(E)|} (\ell-1)^3$, where $\mu(E)$ is the group of roots of unity in $E$. Since $[E:\mathbb{Q}] =4$, it is easy to see that $|\mu(E)| \leq 12$, whence
\[
\ell -1 \geq [E^*(\mu_\ell):E^*] = [E^*(A[\ell]):E^*] \geq \frac{1}{12} (\ell-1)^3,
\]
i.e.~$\ell \leq 3$, a contradiction.
\end{proof}

\begin{remark}
It is interesting to notice that if we only consider \textit{absolutely simple} abelian surfaces over $\mathbb{Q}$, then the value $61$ obtained in the course of the previous proof is optimal, as the following example shows. We know from \cite{MR2430994} that there exists an absolutely simple abelian surface $A/\mathbb{Q}$, with good reduction everywhere except at $61$, which admits (potential) complex multiplication by the full ring of integers of $K=\mathbb{Q}\left( \sqrt{-(61+6\sqrt{61})} \right)$.

The discriminant of $K$ is $61^3$, so $K$ is ramified at 61 only, and we have $(61)\mathcal{O}_K = \mathfrak{P}^4$ for a certain prime $\mathfrak{P}$ of $\mathcal{O}_K$. The extension $K/\mathbb{Q}$ is cyclic of degree 4, so -- since it is furthermore unramified outside 61 -- we see by the Kronecker-Weber theorem that it is a sub-extension of $\mathbb{Q}(\mu_{61})/\mathbb{Q}$. Writing $\operatorname{Gal}(K/\mathbb{Q})=\{\operatorname{Id},\sigma,\sigma^2,\sigma^3\}$, the CM type of $A/\mathbb{Q}$ is $\{\operatorname{Id},\sigma\}$, and the reflex norm is $\Phi(x)=x \cdot \sigma^3(x)$. Recall that the reflex norm induces a group morphism $I_K \to I_K$, where $I_K$ is the group of idèles of $K$, by acting on the idèles componentwise. As $K/\mathbb{Q}$ is cyclic, $K$ is its own reflex field, and as a consequence all the endomorphisms of $A$ are defined over $K$.
The class number of $K$ is 1, so if $\omega: I_K \to \operatorname{Gal}\left(K^{\text{ab}}/K\right)$ denotes the reciprocity map of global class field theory we see that $\omega\left( \prod_{v} \mathcal{O}_{K,v}^\times\right)$ is all of $\operatorname{Gal}\left(K^{\text{ab}}/K\right)$. 
Hence, in order to describe the Hecke character $\varepsilon$ attached to $A_K$ it suffices to describe its restriction to $\prod_{v} \mathcal{O}_{K,v}^\times$, and by the explicit construction of \cite[pp.~664 and 667]{MR2430994} we have
\[
\begin{array}{cccc}
\varepsilon :  & \prod_{v} \mathcal{O}_{K,v}^\times & \to & \{\pm 1 \} \\
 & (a_v) & \mapsto & \begin{cases}1, \text{ if }a_{\mathfrak{P}} \text{ is a square in } \mathbb{F}_\mathfrak{P}^\times \\ -1, \text{ otherwise} \end{cases}
\end{array}
\]
By \cite[Corollary 2 to Theorem 5]{MR0236190} we know that, since $\operatorname{End}_K(A)=\mathcal{O}_K$, the representation $\abGal{K} \to \operatorname{Aut} A[61]$ factors as
\[
\abGal{K} \to\abGal{K}^{\operatorname{ab}} \xrightarrow{\; \rho } \left(\mathcal{O}_K \otimes \mathbb{F}_{61}\right)^\times \hookrightarrow \operatorname{Aut} A[61],
\]
and the map $\rho$ can be described on idèle classes as
\[
\rho((a_v)) = \varepsilon((a_v)) \cdot \Phi(a_\mathfrak{P}).
\]
We claim that the image of $\abGal{K}\to \left(\mathcal{O}_K \otimes \mathbb{F}_{61}\right)^\times$ 
is contained in the kernel of the natural map 
$
\left(\mathcal{O}_K \otimes \mathbb{F}_{\mathfrak{P}}\right)^\times \to \mathbb{F}_{\mathfrak{P}}^\times \to \frac{ \mathbb{F}_{\mathfrak{P}}^\times }{ \mathbb{F}_{\mathfrak{P}}^{\times 4}}.
$
 Notice first that if $(a_v)$ is any idèle class, then $\rho((a_v))$ only depends on $a_{\mathfrak{P}}$. Thus to prove our claim it suffices to check that given an element $a_{\mathfrak{P}} \in \mathcal{O}_{K,\mathfrak{P}}^\times$ the product $\varepsilon(a_{\mathfrak{P}}) \Phi(a_{\mathfrak{P}})$ reduces to a fourth power in $\mathbb{F}_{\mathfrak{P}}^\times$. Notice furthermore that $\sigma \in \operatorname{Gal}\left(K/\mathbb{Q}\right)$ acts trivially on $\mathbb{Z}_{61} \subseteq \mathcal{O}_{K,\mathfrak{P}}$, so $\Phi(x)=x\sigma^3(x)$ induces the map $x \mapsto x^2$ on $\mathbb{F}_{\mathfrak{P}}^\times$.
We can now prove our claim. Suppose first that $a_{\mathfrak{P}}$ is a square in $\mathbb{F}_{\mathfrak{P}}^\times$: then we have $\varepsilon(a_{\mathfrak{P}})=1$, and $\varepsilon(a_{\mathfrak{P}}) \Phi(a_{\mathfrak{P}})$ reduces to $1 \cdot (a_{\mathfrak{P}})^2$ in $\mathbb{F}_{\mathfrak{P}}^\times$; since $a_{\mathfrak{P}}$ is a square in $\mathbb{F}_{\mathfrak{P}}^\times$, the product $\varepsilon(a_{\mathfrak{P}}) \Phi(a_{\mathfrak{P}})$ is a fourth power in $\mathbb{F}_{\mathfrak{P}}^\times$ as claimed. Suppose on the other hand that $a_{\mathfrak{P}}$ is not a square in $\mathbb{F}_{\mathfrak{P}}^\times$: then $a_{\mathfrak{P}}^2$ is a square but not a fourth power, and we have $\varepsilon(a_\mathfrak{P})=-1$, which again is a square but not a fourth power in $\mathbb{F}_{\mathfrak{P}}^\times \cong \mathbb{F}_{61}^\times$: the product $\varepsilon(a_\mathfrak{P})\Phi(a_{\mathfrak{P}})$ is then a fourth power in $\mathbb{F}_{\mathfrak{P}}^\times$ as claimed. 
   
Let $d=61^ka$ (with $(61,a)=1$) be the degree of the extension $K(A[61])/K$: by what we just showed, $a$ divides 
$\displaystyle
\left|\ker\left( (\mathcal{O}_K \otimes \mathbb{F}_{\mathfrak{P}})^\times \to \mathbb{F}_{\mathfrak{P}}^\times\bigm/\mathbb{F}_{\mathfrak{P}}^{\times 4} \right) \right|=\left|\mathbb{F}_{\mathfrak{P}}^{\times 4} \right| \times \left| \mathbb{F}_{\mathfrak{P}} \right|^3=15 \cdot 61^3
$, so $a \bigm| 15$.
Then since $[K(\mu_{61}):K] \geq \frac{1}{[K:\mathbb{Q}]} \varphi(61)=15$ and $K(\mu_{61})$ is contained in $K(A[61])$, we see that $[K(\mu_{61}):K]=15$ and $K(A[61])/K(\mu_{61})$ is a pro-61 extension. Finally, since $K$ is contained in $\mathbb{Q}(\mu_{61})$, we have $K(\mu_{61})=\mathbb{Q}(\mu_{61})$ and $K(A[61])=\mathbb{Q}(A[61])$, and therefore $\mathbb{Q}(A[61])/\mathbb{Q}(\mu_{61})$ is a pro-$61$ extension. This shows, as claimed, that the constant 61 is optimal for absolutely simple abelian surfaces with CM.
\end{remark}

As a final remark, we note that the computation of an explicit value for $C(2,2)$ might be within reach with the current state of knowledge on quartic CM fields, and there is work in progress related to the determination of the set $\mathcal{R}(2,2)$, see for example \cite{MR3376741} and \cite{2015arXiv151104869K}.

\medskip

\noindent\textbf{Acknowledgements.} \textit{I thank Abbey Bourdon for an interesting conversation that prompted me to look into this problem.}

\bibliographystyle{plain}
\bibliography{Bibliography}

\end{document}